\documentclass[
    11pt,
    reqno 
]{amsart}

\usepackage{
    amsfonts,
    amsmath,
    amssymb,
    amsthm
}

\usepackage{
    cancel, 
    commath, 
    eucal,
    mathrsfs, 
    mathtools,
    nccmath, 
    stmaryrd, 
    tikz,
    tikz-cd
}

\usetikzlibrary{
    decorations.pathreplacing,
    patterns
}

\newcommand{\bbN}{\mathbb{N}}
\newcommand{\bbZ}{\mathbb{Z}}
\newcommand{\bbC}{\mathbb{C}}

\newcommand{\frakS}{\mathfrak{S}}

\newcommand{\calP}{\mathcal{P}}

\newcommand{\clambda}{\lambda^t} 

\renewcommand{\phi}{\varphi}

\newcommand{\sseq}{\subseteq}
\newcommand{\sm}{\smallsetminus}

\DeclareMathOperator{\fraksl}{\mathfrak{sl}}

\newcommand{\longest}{w_0^n}

\DeclareMathOperator{\dom}{\sf{dom}}

\DeclareMathOperator{\coInv}{coInv}
\DeclareMathOperator{\Inv}{Inv}

\DeclareMathOperator{\slide}{\sf{slide}}
\DeclareMathOperator{\swap}{\sf{swap}}

\newcommand{\row}{\mathsf{i}}
\newcommand{\col}{\mathsf{j}}
\newcommand{\north}{\mathsf{n}}
\newcommand{\east }{\mathsf{e}}
\newcommand{\south}{\mathsf{s}}
\newcommand{\west }{\mathsf{w}}

\DeclareMathOperator{\PD}{\sf{PD}} 
\DeclareMathOperator{\DeltaPD}{\Delta^\pi\sf{PD}}

\newenvironment{shift}[1][(0,0)]{
    \begin{scope}[shift={#1}]
}{
    \end{scope}
}

\newtheorem{theorem}{Theorem}[section]
\newtheorem{proposition}[theorem]{Proposition}

\newtheorem{lemma}[theorem]{Lemma}
\newtheorem{corollary}[theorem]{Corollary}

\theoremstyle{definition}
\newtheorem{definition}[theorem]{Definition}

\theoremstyle{remark}
\newtheorem*{remark}{Remark}

\pgfdeclarelayer{background}
\pgfdeclarelayer{foreground}
\pgfdeclarelayer{important}
\pgfsetlayers{background,main,foreground,important}

\definecolor{OSU_scarlet}{RGB}{186,12,47} 
\definecolor{OSU_gray40}{RGB}{100,106,110}
\definecolor{OSU_thun}{RGB}{65,182,230}
\definecolor{OSU_pink}{RGB}{251,99,126}

\newcommand{\alertcolor}{OSU_thun}
\newcommand{\shadecolor}{OSU_gray40}

\newcommand{\drawSquare}[3][]{
    \begin{pgfonlayer}{background}
    \draw[\shadecolor,thin]
    (#2-1/2,#3+1/2) rectangle (#2+1/2,#3-1/2);
    \end{pgfonlayer}
}
\newcommand{\drawShade}[3][]{
    \begin{pgfonlayer}{foreground}
    \fill[
        \shadecolor,
        opacity=0.5,
        pattern=north west lines
    ]
    (#2-1/2,#3+1/2) rectangle (#2+1/2,#3-1/2);
    \end{pgfonlayer}
}
\newcommand{\drawMark}[3][\markcolor]{
    \begin{pgfonlayer}{important}
    \draw[#1,very thick]
    (#2,#3) circle (1/2);
    \end{pgfonlayer}
}
\newcommand{\drawBox}[3][\shadecolor]{
    \begin{pgfonlayer}{background}
    \fill[#1,opacity=0.4]
    (#2-5/16,#3+5/16) rectangle +(5/8,-5/8);
    \end{pgfonlayer}
}

\newcommand{\drawHorizontal}[3][\pipecolor,\pipewidth]{
    \draw[#1]
    (#2-1/2,#3) to (#2+1/2,#3);
}
\newcommand{\drawVertical}[3][\pipecolor,\pipewidth]{
    \draw[#1]
    (#2,#3-1/2) to (#2,#3+1/2);
}
\newcommand{\drawRelbow}[3][\pipecolor,\pipewidth]{
    \draw[#1]
    (#2,#3-1/2) to (#2,#3-1/4) to[bend left] (#2+1/4,#3) to (#2+1/2,#3);
}
\newcommand{\drawJelbow}[3][\pipecolor,\pipewidth]{
    \draw[#1]
    (#2-1/2,#3) to (#2-1/4,#3) to[bend right] (#2,#3+1/4) to (#2,#3+1/2);
}
\newcommand{\drawCross}[3][\pipecolor,\pipewidth]{
    \drawHorizontal[#1]{#2}{#3}
    \drawVertical[#1]{#2}{#3}
}
\newcommand{\drawBump}[3][\pipecolor,\pipewidth]{
    \drawRelbow[#1]{#2}{#3}
    \drawJelbow[#1]{#2}{#3}
}

\makeatletter

\newcount\PD@a
\newcount\PD@b
\newcount\PD@c

\newcount\PD@i
\newcount\PD@j

\newcount\PD@m
\newcount\PD@n

\newcount\PD@mode

\newcommand{\drawPDTile}[3]{
    \ifnum #3=0 \drawBump{#1}{#2}  \fi
    \ifnum #3=1 \drawCross{#1}{#2} \fi
    \ifnum #3=2 \drawBox{#1}{#2} \drawBump{#1}{#2}  \fi
    \ifnum #3=3 \drawBox{#1}{#2} \drawCross{#1}{#2} \fi
    \ifnum #3=4 \drawBump{#1}{#2} \drawMark{#1}{#2} \fi
}

\newcommand{\drawPD}[2]{
    
    \PD@i = 1
    \PD@j = 1

    \foreach \t in {#2}{
        \drawPDTile{\number\PD@j}{-\number\PD@i}{\t}
        \drawSquare{\number\PD@j}{-\number\PD@i}
        
        \ifnum \numexpr\number\PD@i+\number\PD@j\relax = #1
            \global\advance \PD@i by 1
            \global\PD@j = 1
        \else
            \global\advance \PD@j by 1
        \fi
    }

    \foreach \i in {1,...,#1}{
        \drawJelbow{#1+1-\i}{-\i}
        \drawSquare{#1+1-\i}{-\i}
    }
}

\newcommand{\drawPDShade}[2]{
    \PD@i = 1
    \PD@j = 1
    \foreach \t in {#2}{
        \ifnum \t>0 \drawShade{\number\PD@j}{-\number\PD@i} \fi

        \begin{pgfonlayer}{foreground}
        \ifnum \t=2
            \draw[\shadecolor,very thick] (\number\PD@j-1/2,-\number\PD@i-1/2) -- ++(1,0);
        \fi
        \ifnum \t=3
            \draw[\shadecolor,very thick] (\number\PD@j+1/2,-\number\PD@i+1/2) -- ++(0,-1);
        \fi
        \ifnum \t=4
            \draw[\shadecolor,very thick] (\number\PD@j-1/2,-\number\PD@i-1/2) -- ++(1,0);
            \draw[\shadecolor,very thick] (\number\PD@j+1/2,-\number\PD@i+1/2) -- ++(0,-1);
        \fi
        \end{pgfonlayer}
        
        \ifnum \numexpr\number\PD@i+\number\PD@j\relax = #1
            \global\advance \PD@i by 1
            \global\PD@j = 1
        \else
            \global\advance \PD@j by 1
        \fi
    }
}

\newcommand{\drawAxes}[1]{
    \begin{pgfonlayer}{foreground}
    \draw[black,very thick]
    (1/2,-#1-1/2) -- (1/2,-1/2) -- (#1+1/2,-1/2);
    \end{pgfonlayer}
}

\newcommand{\drawNorthLabels}[3]{
    \PD@j = 0
    \foreach \j in {#3}{
        \node at (#2+\PD@j,-#1) {\j};
        \global\advance\PD@j by 1
    }
}

\newcommand{\drawWestLabels}[3]{
    \PD@i = 0
    \foreach \i in {#3}{
        \node at (#2,-#1-\PD@i) {\i};
        \global\advance\PD@i by 1
    }
}

\newcommand{\drawPipe}[4]{
    \foreach \i/\j in {#1}{
        \drawHorizontal[\alertcolor,thick,double]{\j}{-\i}
    }
    \foreach \i/\j in {#2}{
        \drawVertical[\alertcolor,thick,double]{\j}{-\i}
    }
    \foreach \i/\j in {#3}{
        \drawRelbow[\alertcolor,thick,double]{\j}{-\i}
    }
    \foreach \i/\j in {#4}{
        \drawJelbow[\alertcolor,thick,double]{\j}{-\i}
    }
}

\makeatother

\newcommand{\smalltilescale}{.22}

\newcommand{\bigtilescale}{.5}

\newcommand{\smalltilebaseline}{4.2pt}

\newcommand{\cross}[1][scale=\tilescale,baseline=\tilebaseline]{
    \begin{tikzpicture}[#1]
    \drawSquare11
    \drawCross11
    \end{tikzpicture}
}

\newcommand{\jelbow}[1][scale=\tilescale,baseline=\tilebaseline]{
    \begin{tikzpicture}[#1]
    \drawSquare11
    \drawJelbow11
    \end{tikzpicture}
}
\newcommand{\bump}[1][scale=\tilescale,baseline=\tilebaseline]{
    \begin{tikzpicture}[#1]
    \drawSquare11
    \drawBump11
    \end{tikzpicture}
}

\newcommand{\smallcross}{\cross[scale=\smalltilescale,baseline=\smalltilebaseline]}

\newcommand{\smallbump}{\bump[scale=\smalltilescale,baseline=\smalltilebaseline]}

\newcommand{\bigcross}{\cross[scale=\bigtilescale]}

\newcommand{\bigjelbow}{\jelbow[scale=\bigtilescale]}
\newcommand{\bigbump}{\bump[scale=\bigtilescale]}

\title{Derivatives of padded Schubert polynomials through pipe dreams}

\author{Hugh Dennin}
\thanks{The author was partially supported by NSF awards DMS-2231565 and DMS-1945212.}

\date{\today}

\begin{document}

\begin{abstract}
    In recent work, Hamaker, Pechenik, Speyer, and Weigandt showed that a certain differential operator $\nabla$ expands positively in the basis of Schubert polynomials.
    For $\pi$ a dominant permutation, Gaetz and Tung showed an analogous positivity result holds for a dual differential operator $\Delta$ acting on $\pi$-padded Schubert polynomials.
    In this paper, we provide a new proof of the result of Gaetz and Tung using the combinatorics of pipe dreams.
\end{abstract}

\maketitle

\section{Introduction}

Fix a partition $\lambda = (\lambda_1\geq \lambda_2\geq\cdots)$ and write $\pi = \dom(\lambda)$ for the dominant permutation with code $\lambda$.
The differential operators \[
    \Delta := \sum_{i=1}^\infty x_i\frac{\partial}{\partial y_i}
    \qquad\text{and}\qquad
    \nabla := \sum_{i=1}^\infty y_i\frac{\partial}{\partial x_i}
\] act on the subspace $V_\lambda$ of the polynomial ring $\bbZ[x_1,x_2,\dots;y_1,y_2,\dots]$ spanned by the monomials $x^\alpha y^\beta = \prod_{i=1}^\infty x_i^{\alpha_i}y_i^{\beta_i}$ satisfying $\alpha_i+\beta_i = \lambda_i$ for each $i\geq 1$.

When $w$ lies below $\pi$ in (left) weak order, the monomials appearing in the \textit{Schubert polynomial} $\frakS_w(x)\in \bbZ_{\geq0}[x_1,x_2,\dots]$ each divide $x^\lambda$.
Hence we may consider the homogenization $\frakS_w^\pi(x;y) := y^\lambda\frakS_w^\pi(\frac{x_1}{y_1},\frac{x_2}{y_2},\dots)\in V_\lambda$, called the \textit{$\pi$-padded Schubert polynomial} of $w$.

The goal of this paper is to provide a combinatorial proof of the following identity of Gaetz and Tung \cite{GaetzTung21} showing that $\Delta$ admits a positive expansion when acting on a $\pi$-padded Schubert polynomial:

\begin{theorem}[$\Delta$ identity \cite{GaetzTung21}]
\label{thm:delta pi}
Let $\pi = \dom(\lambda)$ be a dominant permutation.
For $w\leq_L \pi$, we have \[
    \Delta\frakS_w^\pi(x;y)
    = \sum_{w\lessdot_\pi t_{ab}w} 
    \left(1+|A(w,t_{ab}w)|+|B(w,t_{ab}w)|\right)
    \frakS_{t_{ab}w}^\pi(x;y).
\]
Here $w\lessdot_\pi t_{ab}w$ is shorthand for $w\lessdot t_{ab}w\leq_L \pi$ and \begin{align*}
    A(w,t_{ab}w)
    &:= \{k > b \mid w^{-1}(a) < w^{-1}(k) < w^{-1}(b)\}, \\
    B(w,t_{ab}w)
    &:= \{k > b \mid \pi w^{-1}(b) < \pi w^{-1}(k) < \pi w^{-1}(a)\}.
\end{align*}
\end{theorem}

\noindent
We remark that the special case of Theorem \ref{thm:delta pi} where $\pi = \longest = n\dots21$ is the longest permutation in $S_n$ was originally proved by Gaetz and Gao \cite{GaetzGao20a}.

To motivate Theorem \ref{thm:delta pi}, it is well-known that $\Delta$ and $\nabla$ act as the raising and lowering operators of a $\fraksl_2(\bbC)$-representation on $\bbC\otimes_\bbZ V_\lambda$.
In other words, letting $H$ denote the commutator $[\Delta,\nabla] = \Delta\nabla - \nabla\Delta$,
one can verify the relations \[
    [\Delta,\nabla] = H,
    \qquad
    [H,\Delta] = 2\Delta,
    \qquad
    [H,\nabla] = -2\nabla.
\]
An earlier result of Hamaker, Pechenik, Speyer, and Weigandt \cite{HamakerPechenikSpeyerWeigandt20} shows that $\nabla$ also expands positively when applied to a $\pi$-padded Schubert polynomial:

\begin{theorem}[$\nabla$ identity \cite{HamakerPechenikSpeyerWeigandt20}]
\label{thm:nabla}
For $w\in S_\infty$, we have \[
    \nabla\frakS^\pi_w(x;y) =
    \sum_{s_kw\lessdot_L w} k\cdot \frakS^\pi_{s_kw}(x;y).
\]
\end{theorem}

\noindent
Together, Theorems \ref{thm:delta pi} and \ref{thm:nabla} show that the subspace with basis given by the $\pi$-padded Schubert polynomials $\frakS^\pi_w(x;y)$ for $w\leq_L \pi$ constitutes an $\fraksl_2(\bbC)$-subrepresentation of $\bbC\otimes_\bbZ V_\lambda$.
In particular, this implies that the interval $[e,\pi]_L$ in weak order satisfies the strong Sperner property \cite{GaetzTung21}.

The Schubert polynomial $\frakS_w(x)$ is known to enumerate a class of wiring diagrams associated to $w$ called \textit{pipe dreams} \cite{FominKirillov96,BergeronBilley93,KnutsonMiller05}, with each pipe dream $P$ contributing a monomial $x^{P(\smallcross)} = \prod_{(i,j)\in P(\smallcross)} x_i$ determined by the locations of the $\cross$ tiles in $P$.
In light of this, the left hand side of Theorem \ref{thm:nabla} enumerates pipe dreams for $w$ paired with a distinguished $\cross$ tile, with the content of the theorem implying that there exists a weight-preserving (many-to-one) correspondence between these pairs and pipe dreams for weak order covers $s_kw\lessdot w$.
Such a correspondence was provided by the author in \cite{Dennin25}.

In order to get an analogous combinatorial interpretation of Theorem \ref{thm:delta pi}, we provide a new pipe dream formula for $\frakS^\pi_w(x;y)$ where each pipe dream $P$ contributes a weight $x^{P(\smallcross)} y^{P(\smallbump)\cap P(\pi)}$ with the $y$-variables coming from a particular subset of the $\bump$ tiles of $P$ determined by $\pi$.
The left hand side of Theorem \ref{thm:delta pi} then enumerates pipe dreams for $w$ paired with a distinguished $\bump$ taken from this subset.
We then construct a weight-preserving correspondence $\Phi$ between these pairs and pipe dreams for permutations $w\lessdot_\pi t_{ab}w$, thus giving a combinatorial proof of Theorem \ref{thm:delta pi}.

The organization of this paper is as follows:
in Section 2, we cover the necessary background on permutations and pipe dreams.
In section 3, we provide the new pipe dream formula for $\pi$-padded Schubert polynomials.
Finally, in Section 4, we construct the correspondence $\Phi$ between the pipe dreams counted by both sides of Theorem \ref{thm:delta pi} and show that it has the desired properties.

\section{Background \& Definitions}

    
\subsection{Permutations, Bruhat Order, and Weak Order}
\label{sect:perm}
Let $S_n$ denote the collection of permutations of $\bbN := \{1,2,\dots\}$ with \textit{support} $\{i\geq 1\mid w(i)\neq i\}$ contained in $[n] := \{1,\dots,n\}$, and let $S_\infty := \bigcup_{n\geq 1} S_n$ denote the set of permutations of $\bbN$ with finite support.
We will frequently write permutations $w\in S_n$ using their \textit{one-line notation} $w(1)w(2)\dots w(n)$.

For $a,b\geq 1$ with $a<b$, write $t_{ab}\in S_\infty$ for the transposition that swaps $a$ and $b$.
$S_\infty$ is generated as a group by the simple transpositions $s_k := t_{k\,k+1}$ for $k\geq 1$.
The \textit{length} of a permutation $w\in S_\infty$ is the smallest number $\ell(w)\geq 0$ such that $w$ can be written as a product of $\ell(w)$ simple transpositions.

The \textit{Bruhat order} (or \textit{strong order}) on $S_\infty$ is the partial order $\leq$ obtained as the transitive closure of the cover relations $w\lessdot t_{ab}w$ whenever $\ell(t_{ab}w) = \ell(w)+1$.
The suborder $\leq_L$ of Bruhat order obtained by only considering covers of the form $w\lessdot s_kw$ where $\ell(s_kw) = \ell(w)+1$ is called the \textit{(left) weak order} on $S_n$.
For permutations $w\leq_L \pi$, we use the shorthand $w\lessdot_\pi t_{ab}w$ to denote that both $w\lessdot t_{ab}w$ and $t_{ab}w\leq_L \pi$.

The \textit{inversion set} of a permutation $w$ is the collection of pairs \[
    \Inv(w) = \{(i,j)\in \bbN^2  \mid i < j\text{ and }w(i) > w(j)\}.
\]
It is well-known that $\ell(w) = |\Inv(w)|$.
A useful transformation of the inversion set is the \textit{diagram} of $w$, given by
\begin{align*}
    D(w) &= \{(i,w(j)) \mid (i,j)\in \Inv(w)\} \\
    &= \{(i,j)\in \bbN^2  \mid i < w^{-1}(j)\text{ and }j < w(i)\}.
\end{align*}
We will also make use of \textit{co-inversion set} defined as \[
    \coInv(w) = \{(i,j)\in \bbN^2  \mid i<j\text{ and }w(i)<w(j)\}.
\]

The weak order on $S_\infty$ can be characterized in terms of inversion sets.
\begin{proposition}
\label{prop:weak order}
Let $w,u\in S_\infty$.
The following are equivalent:
\begin{itemize}
    \item[(i)] $w\leq_L u$,
    \item[(ii)] $\Inv(w)\sseq \Inv(u)$ (or equivalently $\coInv(u)\sseq \coInv(w)$),
    \item[(iii)] $\Inv(w^{-1})\sseq \coInv(uw^{-1})$ (or equivalently $\Inv(u w^{-1}) \sseq \coInv(w^{-1})$).
\end{itemize}
\end{proposition}

    
\subsection{Pipe Dreams}

In what follows, we will visualize the positive quadrant $\bbN^2$ using matrix coordinates, so $(i,j)\in \bbN^2$ will refer to the position $i$\textsuperscript{th} from the top and $j$\textsuperscript{th} from the left.
We will often denote positions $(i,j)$ using a single letter, e.g. $p$, opting to write $\row_p = i$ and $\col_p = j$ instead.

\begin{definition}[\cite{BergeronBilley93,KnutsonMiller05}]
    A \textit{pipe dream}\footnote{
        Also known as an \textit{RC-graph}.
    } $P$ is a finite subset of $\bbN^2$, which we interpret as a tiling of $\bbN^2$ using \textit{bump} and \textit{cross} tiles \[
        \bigbump \qquad \bigcross
    \] with $P$ giving the location of the $\cross$ tiles.
    The resulting diagram consists of countably many pipes winding from the north edge to the west edge.
    
    All pipe dreams we will encounter will be assumed to be \textit{reduced}, meaning that no two pipes may cross one another more than once.
\end{definition}

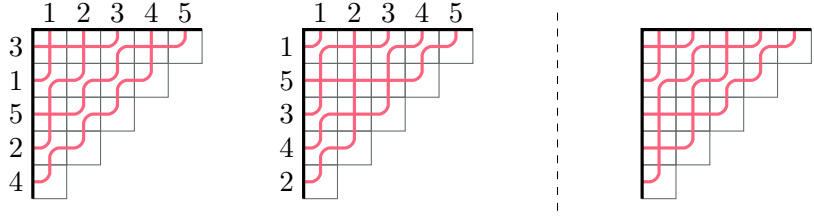
\begin{figure}
\begin{tikzpicture}[scale=0.45]
    \drawPD{5}{
        1,1,0,1,
        0,0,0,
        1,0,
        0
    }
    \drawAxes{5}
    \drawNorthLabels{0}{1}{1,2,3,4,5}
    \drawWestLabels{1}{0}{3,1,5,2,4}

    \begin{shift}[(8,0)]
    \drawPD{5}{
        0,1,0,0,
        1,1,1,
        0,1,
        0
    }
    \drawAxes{5}
    \drawNorthLabels{0}{1}{1,2,3,4,5}
    \drawWestLabels{1}{0}{1,5,3,4,2}
    \end{shift}
    
    \draw[black,dashed] (16,0) -- (16,-6);
    
    \begin{shift}[(18,0)]
    \drawPD{5}{
        1,0,1,0,
        0,0,0,
        1,1,
        1
    }
    \drawAxes{5}
    \end{shift}
\end{tikzpicture}
\caption{
    Left: Two (reduced) pipe dreams with permutations 31524 and 15342.
    Right: A non-reduced pipe dream.
}
\label{fig:pd examples}
\end{figure}

Given a pipe dream $P$, write (for emphasis) $P(\cross) = P$ and $P(\bump) = \bbN^2 \sm P$ for the sets of cross and bump tiles in $P$, respectively.
We label each pipe in $P$ with the index of the column it reaches along the north edge, and will speak of ``pipe $k$ in $P$'' to refer to the pipe in $P$ with label $k$.

For each position $p\in \bbN^2$, write $\north_{P,p}$ (resp. $\east_{P,p}$, $\south_{P,p}$, $\west_{P,p}$) for the label of the pipe exiting position $p$ in $P$ to the north (resp. east, south, west).
If $P$ is understood from context, we will often omit it and just write $\north_p$ (resp. $\east_p$, $\south_p$, $\west_p$).

If all the crossings in a pipe dream $P$ are among the first $n$ pipes, then $P(\cross)$ is contained in the \textit{staircase} $\delta_n := \{(i,j) \mid i+j\leq n\}$.
In this case, we will restrict our visualization of $P$ to $\delta_{n+1}$, drawing $P$ as normal in $\delta_n$ and replacing the $\bump$ tiles along $\delta_{n+1}\sm \delta_n$ with \textit{elbow} tiles \[
    \bigjelbow
\]

The \textit{permutation} of a pipe dream $P$ is the permutation $w \in S_\infty$ obtained by reading off the labels of the pipes as they appear along the west edge of the diagram---this gives $w$ in one-line notation (see Figure \ref{fig:pd examples}).
In other words, for each $a\geq 1$, pipe $a$ will connect column $a$ to row $w^{-1}(a)$.
Two pipes $a < b$ cross in $P$ if and only if $(a,b)\in \Inv(w^{-1})$ if and only if $(w^{-1}(b),w^{-1}(a)) \in \Inv(w)$.
Write $\PD(w)$ for the set of pipe dreams $P$ with permutation $w$.

We close this section with two moves that can be performed on pairs $(P,p)$ where $P\in \PD(w)$ and $p\in P(\bump)$.

\begin{definition}
\label{def:slide & swap}
Let $P\in \PD(w)$ and $p\in P(\bump)$.
\begin{itemize}
    \item
    The pair $(P,p)$ is \textit{slidable} if $w^{-1}(\west_p) > \row_q$.
    This means there is a position $q\in P(\bump)$ with $\row_q = \row_p$ and $\south_q = \west_p$.
    In this case, write $\slide(P,p) = (P,q)$.

    \item 
    The pair $(P,p)$ is \textit{swappable} if pipes $\west_p$ and $\south_p$ cross at some $q\in P(\cross)$.
    Equivalently, $(P,p)$ is swappable if either $\west_p > \south_p$ or $w^{-1}(\west_p) > w^{-1}(\south_p)$.
    In this case, write $\swap(P,p) = (Q,q)$ where $Q\in \PD(w)$ is the pipe dream obtained from $P$ by swapping the $\bump$ and $\cross$ at positions $p$ and $q$.
\end{itemize}
\end{definition}

See Figure \ref{fig:slide swap} for illustrations of $\slide(P,p)$ and $\swap(P,p)$.

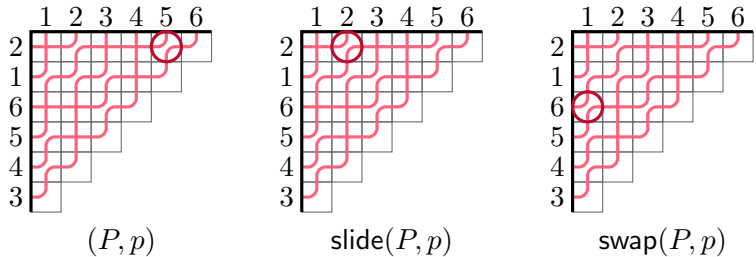
\begin{figure}
\begin{tikzpicture}[scale=0.4]
    \drawPD{6}{
        1,0,1,1,4,
        0,0,0,1,
        1,1,0,
        0,1,
        0
    }
    \drawAxes{6}
    \drawNorthLabels{0}{1}{1,2,3,4,5,6}
    \drawWestLabels{1}{0}{2,1,6,5,4,3}
    \node at (3.5,-7.5) {$(P,p)$};

    \begin{shift}[(9,0)]
    \drawPD{6}{
        1,4,1,1,0,
        0,0,0,1,
        1,1,0,
        0,1,
        0
    }
    \drawAxes{6}
    \drawNorthLabels{0}{1}{1,2,3,4,5,6}
    \drawWestLabels{1}{0}{2,1,6,5,4,3}
    \node at (3.5,-7.5) {$\slide(P,p)$};
    \end{shift}
    
    \begin{shift}[(18,0)]
    \drawPD{6}{
        1,0,1,1,0,
        0,0,0,1,
        4,1,0,
        0,1,
        0
    }
    \drawAxes{6}
    \drawNorthLabels{0}{1}{1,2,3,4,5,6}
    \drawWestLabels{1}{0}{2,1,6,5,4,3}
    \node at (3.5,-7.5) {$\swap(P,p)$};
    \end{shift}
\end{tikzpicture}
\caption{
    $\slide(P,p)$ and $\swap(P,p)$ for a pipe dream $P$ and a bump $p\in P(\bump)$.
}
\label{fig:slide swap}
\end{figure}

\subsection{Partitions and Dominant Permutations}

By a \textit{partition}, we mean a weakly decreasing sequence of nonnegative integers $\lambda = (\lambda_1\geq \lambda_2\geq \cdots)$ with $|\lambda| := \sum_{i=1}^\infty \lambda_i$ finite.
The combinatorics of partitions and permutations are related in many ways, one being the notion of a dominant permutation.

\begin{definition}
    \label{def:dominant}
    A permutation $\pi\in S_n$ is \textit{dominant} if it satisfies any of the following equivalent conditions:\footnote{
    See, say, \cite{Macdonald91,BergeronBilley93} for the equivalence of these conditions.
    }    
    \begin{enumerate}
    \item[(i)] $D(\pi) = \{(i,j)\in \bbN^2 \mid j\leq \lambda_i\}$ is the \textit{diagram} of a partition $\lambda$,
    \item[(ii)] $\pi$ has a unique pipe dream whose $\cross$ tiles are given by $D(\pi)$.
    \item[(iii)] $\pi$ contains no \textit{$132$-patterns} i.e. there is no triple of indices $a<b<c$ with $\pi(a)<\pi(c)<\pi(b)$.
    \end{enumerate}
    In this case, we write $\pi = \dom(\lambda)$. 
\end{definition}

If $\pi = \dom(\lambda)$ is dominant, then $\pi^{-1} = \dom(\clambda)$ is also dominant.
Here, $\clambda$ denotes the \textit{conjugate partition} to $\lambda$ given by $\clambda_j = |\{i\geq 1 \mid \lambda_i \geq j\}|$.
This follows since $D(\pi^{-1})$ is the transpose of $D(\pi)$.

We end this section with a lemma characterizing when covers in Bruhat order are subsumed by a dominant permutation in weak order.

\begin{proposition}
\label{prop:cover lemma}
Let $\pi\in S_\infty$ be dominant and assume both $w\leq_L \pi$ and $w \lessdot t_{ab}w$.
Then $t_{ab}w\leq_L \pi$ if and only if $(w^{-1}(a),w^{-1}(b))\in \Inv(\pi)$ (or equivalently, $(a,b)\in \Inv(\pi w^{-1})$).
\end{proposition}

\noindent
The proof is an exercise using Proposition \ref{prop:weak order} and is left to the interested reader.
The forwards implication does not require the dominant hypothesis.
For the backwards direction, one way to proceed is use the relationship between $\Inv(\pi)$ and $D(\pi)$, the latter being the diagram of a partition since $\pi$ is dominant.


\subsection{Schubert Polynomials}
\label{sect:polys}
Given a finite subset $S\sseq \bbN^2$, we associate the monomial \[
    x^S := \prod_{(i,j)\in S} x_i \in \bbZ[x_1,x_2,\dots].
\]

The \emph{Schubert polynomial} for $w$ is the generating function \[
    \frakS_w(x) :=
    \sum_{P\in \PD(w)} x^{P(\smallcross)}.
\] enumerating pipe dreams for $w$ \cite{FominKirillov96,BergeronBilley93}.
Schubert polynomials were first introduced by Lascoux and Schutzenberger \cite{LascouxSchuztenberger82a} to represent Schubert cycles in the cohomology ring of the complete flag variety.

If $\pi = \dom(\lambda)$ is a dominant permutation, then $\frakS_\pi(x)$ is the single monomial $x^\lambda = \prod_{i=1}^\infty x_i^{\lambda_i}$.
If $w\leq_L \pi$, then the monomials appearing in $\frakS_w(x)$ all divide $x^\lambda$ (e.g. by Theorem \ref{thm:nabla}).
Based on this, we can obtain the \emph{$\pi$-padded Schubert polynomial} \cite{GaetzGao20a,GaetzTung21} as the homogenization of $\frakS_w(x)$ given by \[
    \frakS^\pi_w(x;y) := y^\lambda \frakS_w(\mfrac{x_1}{y_1},\mfrac{x_2}{y_2},\dots).
\]
This is a homogenization in the sense that for each $i\geq 1$, the sum of the $x_i$-degree and $y_i$-degree of each monomial appearing in $\frakS^\pi_w(x;y)$ is $\lambda_i$.

\section{A pipe dream formula for $\frakS^\pi_w(x;y)$}

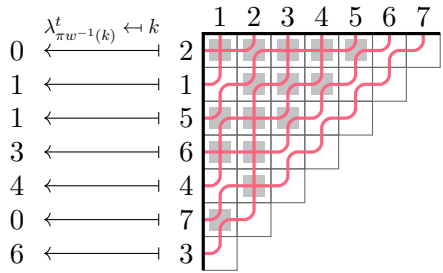
\begin{figure}
\begin{tikzpicture}[scale=0.45]
    \drawPD{7}{
        3,2,3,3,2,0,
        0,2,2,2,0,
        2,3,2,0,
        3,3,0,
        0,3,
        2
    }
    \drawAxes{7}
    \drawNorthLabels{0}{1}{1,2,3,4,5,6,7}
    
    \drawWestLabels {1}{0}{2,1,5,6,4,7,3}
    \draw node at (-2.5,-0.4) {
        \scalebox{0.7}{$\clambda_{\pi w^{-1}(k)}\mapsfrom k$}
    };
    \foreach \i in {1,...,7}{
        \draw[|->] (-0.8,-\i) -- (-4.2,-\i);
    }
    \drawWestLabels { 1}{-5}{0,1,1,3,4,0,6}
\end{tikzpicture}
\caption{
    The $\pi$-dominated positions of a pipe dream $P\in \PD(w)$ for $w = 2156473$ and $\pi = 645371$ ($\clambda = (6,4,3,1,1,0,0)$).
}
\label{fig:dominant pos}
\end{figure}

Consider the longest permutation $\longest := n\dots 21$ in $S_n$, which is the dominant permutation corresponding to the partition $\lambda = (n-1,\dots,2,1)$.
Given $w\in S_n$ (i.e. $w\leq_L \longest$), notice that the $\longest$-padded Schubert polynomial for $w$ can be expressed as
\[
    \frakS^{\longest}_w(x;y)
    = \sum_{P\in\PD(w)} x^{P(\smallcross)} y^{P(\smallbump) \cap \delta_n}.
\]
This follows since the staircase $\delta_n$ consists of $\lambda_i = n-i$ positions in each row $1\leq i\leq n$ and contains $P(\cross)$.
The benefit of the above formula is that we get a direct combinatorial interpretation of the $y$-variables appearing in $\frakS^{\longest}_w(x;y)$ as corresponding to specific $\bump$ tiles from the pipe dreams for $w$.
This will be useful later when examining the left hand side of Theorem \ref{thm:delta pi}.

The goal of this section is to provide an analogous pipe dream formula for $\frakS^\pi_w(x;y)$ for a general dominant permutation $\pi$.
Specifically, we will determine for each $P\in \PD(w)$ a distinguished subset $P(\pi)\sseq \bbN^2$ such that \begin{itemize}
    \item[(D.1)] $P(\pi)$ contains exactly $\lambda_i$ positions in each row $i\geq 1$.
    \item[(D.2)] $P(\pi)$ contains $P(\cross)$.
\end{itemize}
Given such an assignment, we can then write \begin{equation}
    \label{eq:padded schubert pipe dream formula}
    \frakS^\pi_w(x;y)
    = \sum_{P\in\PD(w)} x^{P(\smallcross)} y^{P(\smallbump)\cap P(\pi)}.
\end{equation}

\begin{definition}
    \label{def:dominant pos}
    Let $P\in \PD(w)$.
    A position $p$ in $P$ is \textit{dominated} by $\pi$ if $\row_p \leq \clambda_{\pi w^{-1}(\south_p)}$.
    Write $P(\pi)$ for the set of positions in $P$ dominated by $\pi$.
\end{definition}

\begin{remark}
Given indices $i,k\geq 1$, notice that \begin{align*}
    i \leq \clambda_{\pi w^{-1}(k)}
    \iff& \pi w^{-1}(k)\leq \lambda_i, \\
    \iff& (i,\pi w^{-1}(k))\in D(\pi),
    \tag{since $\pi$ is dominant} \\
    \iff& (i,w^{-1}(k))\in \Inv(\pi).
\end{align*}
We will frequently pass between these equivalences without comment.
\end{remark}

\begin{proposition}
\label{prop:dominant pos}
$P(\pi)$ satisfies \textup{(}D.1\textup{)} and \textup{(}D.2\textup{)}.
In particular, the pipe dream formula for $\frakS^\pi_w(x;y)$ given in Equation \ref{eq:padded schubert pipe dream formula} holds.
\end{proposition}

\begin{proof}
Given a fixed pipe $k$ in $P$, for each $1\leq i< w^{-1}(k)$, notice that there is a unique position $p_{k,i}$ with $\row_{p_{k,i}} = i$ and $\south_{p_{k,i}} = k$, and these are the only positions in $P$ where the pipe exiting to the south is $k$.
With this notation, the set of positions in $P$ dominated by $\pi$ can be written as \begin{align*}
    P(\pi) &= \{p_{k,i} \mid k\geq 1, i\in [\clambda_{\pi w^{-1}(k)}]\}, \\
    &= \{p_{w\pi^{-1}(j),i} \mid j\geq 1, i\in [\clambda_j]\}, \\
    &= \{p_{w\pi^{-1}(j),i} \mid i\geq 1, j\in [\lambda_i]\}
\end{align*} with the first equality utilizing that we always have $w^{-1}(k) > \clambda_{\pi w^{-1}(k)}$ since $(w^{-1}(k),w^{-1}(k))\notin \Inv(\pi)$.
This shows (D.1).

For (D.2), suppose $p\in P(\cross)$ is given.
Write $k = \south_p$ and $a = \west_p$ for the pipes in $P$ crossing at $p$, so $k < a$.
Then $(w^{-1}(a),w^{-1}(k))\in \Inv(w) \sseq \Inv(\pi)$ implies $w^{-1}(a) \leq \clambda_{\pi w^{-1}(k)}$.
Since pipe $a$ exits $P$ at or below row $\row_p$, we also get $\row_p \leq w^{-1}(a)$.
\end{proof}

\begin{figure}
\begin{tikzpicture}[scale=0.45]
    \begin{shift}[(0,0)]
    \drawPD{4}{
        2,2,0,
        3,3,
        3
    }
    \drawAxes{4}
    \end{shift}
    
    \begin{shift}[(5,0)]
    \drawPD{4}{
        2,0,3,
        3,2,
        3
    }
    \drawAxes{4}
    \end{shift}
    
    \begin{shift}[(10,0)]
    \drawPD{4}{
        2,3,0,
        3,3,
        2
    }
    \drawAxes{4}
    \end{shift}
    
    \begin{shift}[(15,0)]
    \drawPD{4}{
        0,3,3,
        2,2,
        3
    }
    \drawAxes{4}
    \end{shift}
    
    \begin{shift}[(20,0)]
    \drawPD{4}{
        0,3,3,
        2,3,
        2
    }
    \drawAxes{4}
    \end{shift}
    
    \node at (12.5,-6) {$
        \frakS_{1432}^{3421}(x;y) =
        y_1^2x_2^2x_3 + x_1y_1x_2y_2x_3 + x_1y_1x_2^2y_3 + x_1^2y_2^2x_3 + x_1^2x_2y_2y_3
    $};
\end{tikzpicture}
\caption{
    Computation of $\frakS^\pi_w(x;y)$ using Equation \ref{eq:padded schubert pipe dream formula} for $w = 1432$ and $\pi = 3421$ ($\lambda^t = (3,2,0,0)$).
    The five pipe dreams in $\PD(1432)$ with their $3421$-dominated positions are drawn.
}
\label{fig:padded schub}
\end{figure}

\section{The Correspondence $\Phi$}

In light of the pipe dream formula for $\frakS^\pi_w(x;y)$
(Equation \ref{eq:padded schubert pipe dream formula}),
we can interpret the left-hand side of Theorem \ref{thm:delta pi} as a weighted enumeration of the set \[
    \DeltaPD(w) := 
    \left\{
        (P,p) : P\in \PD(w), p\in P(\bump)\cap P(\pi)
    \right\}
\] where each pair $(P,p)$ contributes $\frac{x_{\row_p}}{y_{\row_p}} x^{P(\smallcross)} y^{P(\smallbump)\cap P(\pi)}$.
Theorem \ref{thm:delta pi} would therefore follow from the existence of a map \[
    \Phi\colon \DeltaPD(w)
    \longrightarrow 
    \bigcup_{w\lessdot_\pi t_{ab}w} \PD(t_{ab}w)
\] such that, given $Q\in \PD(t_{ab}w)$ for some $w\lessdot_\pi t_{ab}w$, 
\begin{itemize}
    \item[($\Phi$.1)] $|\Phi^{-1}(Q)| = 1 + |A(w,t_{ab}w)| + |B(w,t_{ab}w)|$
    \item[($\Phi$.2)] $\frac{x_{\row_p}}{y_{\row_p}} x^{P(\smallcross)} y^{P(\smallbump)\cap P(\pi)}
    = x^{Q(\smallcross)} y^{Q(\smallbump)\cap P(\pi)}$ for each $(P,p)\in \Phi^{-1}(Q)$.
\end{itemize}
In this section, we will build such a map $\Phi$, thus giving a combinatorial proof of Theorem \ref{thm:delta pi}.

More precisely, consider the partition $\DeltaPD(w) = \calP_0\sqcup \calP_A\sqcup \calP_B$ given by \begin{align*}
    \calP_0 &:= \{
        (P,p)\in \DeltaPD(w) : 
        (\west_p,\south_p)\in \Inv(\pi w^{-1})
    \}, \\
    \calP_A &:= \{
        (P,p)\in \DeltaPD(w) : 
        \south_p < \west_p 
    \}, \\
    \calP_B &:= \{
        (P,p)\in \DeltaPD(w) : 
        (\west_p,\south_p)\in \coInv(\pi w^{-1})
    \}.
\end{align*}
We will define $\Phi$ by separately constructing its restriction $\Phi_*$ to each $\calP_*$ (for $*\in \{0,A,B\}$).
Given $w\lessdot_\pi t_{ab}w$ and $Q\in \PD(t_{ab}w)$, we will see that these restrictions satisfy the refinement of ($\Phi$.1) above:
\begin{itemize}
    \item[($\Phi$.1z)] $|\Phi_0^{-1}(Q)| = 1$ (i.e. $\Phi_0$ is a bijection),
    \item[($\Phi$.1a)] $|\Phi_A^{-1}(Q)| = |A(w,t_{ab}w)|$,
    \item[($\Phi$.1b)] $|\Phi_B^{-1}(Q)| = |B(w,t_{ab}w)|$.
\end{itemize}

\subsection{Construction of $\Phi$}

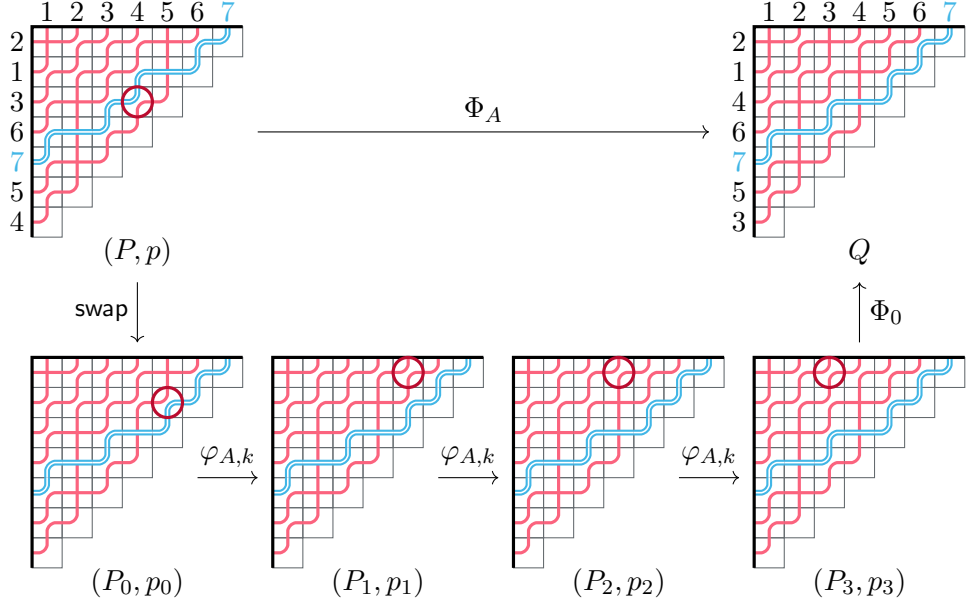
\begin{figure}
\begin{tikzpicture}[scale=0.4]
    \begin{shift}
    \drawPD{7}{
        1,0,0,0,1,0,
        0,0,0,0,1,
        0,1,0,4,
        0,1,0,
        0,1,
        0
    }
    \drawPipe 
    {2/5,4/2} 
    {} 
    {1/6,2/4,3/3,4/1} 
    {1/7,2/6,3/4,4/3,5/1} 
    \drawAxes{7}
    \drawNorthLabels{0}{1}{1,2,3,4,5,6,{\color{\alertcolor}7}}
    \drawWestLabels {1}{0}{2,1,3,6,{\color{\alertcolor}7},5,4}
    
    \node at (4,-8) {$(P,p)$};
    \draw[->] (4,-9) -- (4,-11)
    node[midway,left,scale=0.9] {$\swap$};
    \draw[->] (8,-4) -- (23,-4)
    node[midway,above] {$\Phi_A$};
    \end{shift}

    \begin{shift}[(0,-11)]
    \drawPD{7}{
        1,0,0,0,1,0,
        0,0,0,0,4,
        0,1,0,1,
        0,1,0,
        0,1,
        0
    }
    \drawPipe 
    {3/4,4/2} 
    {} 
    {1/6,2/5,3/3,4/1} 
    {1/7,2/6,3/5,4/3,5/1} 
    \drawAxes{7}

    \node at (4,-8) {$(P_0,p_0)$};
    \draw[->] (6,-4.5) -- (8,-4.5)
    node[midway,above] {$\phi_{A,k}$};
    \end{shift}

    \begin{shift}[(8,-11)]
    \drawPD{7}{
        1,0,0,0,4,0,
        0,0,0,1,0,
        0,1,0,1,
        0,1,0,
        0,1,
        0
    }
    \drawPipe 
    {3/4,4/2} 
    {} 
    {1/6,2/5,3/3,4/1} 
    {1/7,2/6,3/5,4/3,5/1} 
    \drawAxes{7}
    
    \node at (4,-8) {$(P_1,p_1)$};
    \draw[->] (6,-4.5) -- (8,-4.5)
    node[midway,above] {$\phi_{A,k}$};
    \end{shift}

    \begin{shift}[(16,-11)]
    \drawPD{7}{
        1,0,0,4,0,0,
        0,0,0,1,0,
        0,1,0,1,
        0,1,0,
        0,1,
        0
    }
    \drawPipe 
    {3/4,4/2} 
    {} 
    {1/6,2/5,3/3,4/1} 
    {1/7,2/6,3/5,4/3,5/1} 
    \drawAxes{7}

    \node at (4,-8) {$(P_2,p_2)$};
    \draw[->] (6,-4.5) -- (8,-4.5)
    node[midway,above] {$\phi_{A,k}$};
    \end{shift}

    \begin{shift}[(24,-11)]
    \drawPD{7}{
        1,0,4,0,0,0,
        0,0,0,1,0,
        0,1,0,1,
        0,1,0,
        0,1,
        0
    }
    \drawPipe 
    {3/4,4/2} 
    {} 
    {1/6,2/5,3/3,4/1} 
    {1/7,2/6,3/5,4/3,5/1} 
    \drawAxes{7}

    \node at (4,-8) {$(P_3,p_3)$};
    \draw[->] (4,0) -- (4,2)
    node[midway,right] {$\Phi_0$};
    \end{shift}

    \begin{shift}[(24,0)]
    \drawPD{7}{
        1,0,1,0,0,0,
        0,0,0,1,0,
        0,1,0,1,
        0,1,0,
        0,1,
        0
    }
    \drawPipe 
    {3/4,4/2} 
    {} 
    {1/6,2/5,3/3,4/1} 
    {1/7,2/6,3/5,4/3,5/1} 
    \drawAxes{7}
    \drawNorthLabels{0}{1}{1,2,3,4,5,6,{\color{\alertcolor}7}}
    \drawWestLabels {1}{0}{2,1,4,6,{\color{\alertcolor}7},5,3}

    \node at (4,-8) {$Q$};
    \end{shift}
\end{tikzpicture}
\caption{
    Computation of $\Phi_A(P,p)$ for a pair $(P,p)\in \calP_A$ ($w = 2136754$).
    Pipe $k = 7$ is emphasized in blue.
    The final pipe dream $Q$ has permutation $t_{34}w = 2146753$.
}
\label{fig:Phi_A}
\end{figure}

We first detail the bijection $\Phi_0$.
Consider a pair $(P,p)\in \calP_0$.
Then $(\west_p,\south_p)\in \Inv(\pi w^{-1})\sseq \coInv(w^{-1})$, so pipes $a = \west_p$ and $b = \south_p$ do not cross in $P$.
Replacing the $\bump$ at position $p$ in $P$ with a $\cross$ crosses pipes $a$ and $b$, so we obtain a (reduced) pipe dream $Q\in \PD(t_{ab}w)$.
Notice that $\ell(t_{ab}w) = |Q(\cross)| = |P(\cross)|+1 = \ell(w) + 1$, so $w\lessdot t_{ab}w$.
Furthermore, Proposition \ref{prop:cover lemma} guarantees that $w\lessdot_\pi t_{ab}w$.
We may therefore define 

\begin{definition}[The map $\Phi_0$]
\label{def:Phi_0}
Given $(P,p)\in \calP_0$, we set $\Phi_0(P,p) := Q$, the pipe dream obtained by replacing the $\bump$ at position $p$ with a $\cross$.
\end{definition}

We next construct both $\Phi_A$ and $\Phi_B$ in parallel.
The idea will be to develop an procedure that takes a pair appearing in $\calP_A$ or $\calP_B$ and reduces it to a pair in $\calP_0$, from which we can then apply $\Phi_0$.
For now, we focus only on the definition of $\Phi_A$ and $\Phi_B$, postponing the statements and proofs of certain results needed to show that these maps are well-defined until the next section.
We will need the following pieces of terminology:

\begin{definition}
Let $P\in \PD(w)$, $k\geq 1$, and $p\in P(\bump)$.
\begin{itemize}
\item
Say that $p$ is \textit{$(A,k)$-aligned} in $P$ if both
\begin{itemize}
    \item
    $p$ lies (weakly) northwest of pipe $k$,
    \item
    $(\west_p,k) \in \Inv(w^{-1})$.
\end{itemize}

\item 
Say that $p$ is \textit{$(B,k)$-aligned} in $P$ if each of
\begin{itemize}
    \item $p$ lies (weakly) northwest of pipe $k$,
    \item $(\west_p,k) \in \coInv(\pi w^{-1})$,
    \item $\row_p \leq \clambda_{\pi w^{-1}(k)}$ (or equivalently $(\row_p,w^{-1}(k))\in \Inv(\pi)$).
\end{itemize}
\end{itemize}
In the case $p$ is $(*,k)$-aligned in $P$ (where $*\in \{A,B\}$), we will also say that the pair $(P,p)$ is \textit{$(*,k)$-aligned}.
\end{definition}

In the sequel, let the wildcard $*$ denote either $A$ or $B$.
We next define an operator acting on $(*,k)$-aligned pairs which will be used as the fundamental building block of $\Phi_*$.

\begin{definition}
Let $(P,p)$ be a $(*,k)$-aligned pair ($*\in\{A,B\}$, $k\geq 1$). Define $\phi_{*,k}(P,p)$ to be the pair obtained as follows:
\begin{itemize}
    \item
    If $\slide(P,p)$ is both $(*,k)$-aligned and swappable, then set $\phi_{*,k}(P,p) := \swap(\slide(P,p))$.
    \item 
    Otherwise set $\phi_{*,k}(P,p) := \slide(P,p)$.
\end{itemize}
\end{definition}

\noindent
Note that for $\phi_{*,k}(P,p)$ to be well-defined, we need to know that $(*,k)$-aligned pairs are slidable; we will prove this in Lemma \ref{lem:aligned is slidable}.

Now, suppose we are given $(P,p)\in \calP_*$.
We can obtain both an index $k\geq 1$ and a $(*,k)$-aligned pair $(P_0,p_0)$ as follows:
if $(P,p)\in \calP_B$, then $(P_0,p_0) := (P,p)$ itself is $(B,k)$-aligned where $k = \south(P,p)$.
If $(P,p)\in \calP_A$, notice that pipes $\west_p$ and $\south_p$ must cross at a position $p_0\in P(\cross)$ with $\row_{p_0} < \row_p$.
Then $(P_0,p_0) := \swap(P,p)$ is $(A,k)$-aligned, where $k = \west_{P,p} = \south_{P_0,p_0}$.

By repeatedly applying $\phi_{*,k}$ to $(P_0,p_0)$, we obtain a chain of pairs \[
    (P_0,p_0)
    \overset{\phi_{*,k}}\longrightarrow
    (P_1,p_1)
    \overset{\phi_{*,k}}\longrightarrow
    \cdots
    \overset{\phi_{*,k}}\longrightarrow
    (P_m,p_m)
\] where each $(P_i,p_i)$ for $0\leq i\leq m-1$ is $(*,k)$-aligned and $(P_m,p_m)\in \calP_0$ is not $(*,k)$-aligned---this is the content of Corollary \ref{cor:phi terminates}.
We may thus make the definition 

\begin{definition}[The maps $\Phi_A$ and $\Phi_B$]
\label{def:Phi_A/B}
Given $(P,p)\in \calP_*$ (where $*\in \{A,B\}$), we set $\Phi_*(P,p) := \Phi_0(\phi_{*,k}^m(P_0,p_0))$ where
\begin{itemize}
    \item if $* = A$, then $k = \west_{P,p}$ and $(P_0,p_0) = \swap(P,p)$,
    \item if $* = B$, then $k = \south_{P,p}$ and $(P_0,p_0) = (P,p)$,
\end{itemize}
and $m\geq 1$ is taken so that $\phi_{*,k}^m(P_0,p_0)$ is not $(*,k)$-aligned.
\end{definition}

Figures \ref{fig:Phi_A} and \ref{fig:Phi_B} give example computations of $\Phi_A(P,p)$ and $\Phi_B(P,p)$.
In both examples, the pipe dreams $(P_0,p_0)$, $(P_1,p_1)$, and $(P_2,p_2)$ are $(*,k)$-aligned while $(P_3,p_3)$ is not.

\begin{figure}
\begin{tikzpicture}[scale=0.4]
    \begin{shift}[(0,0)]
    \drawPD{7}{
        1,0,0,0,1,0,
        0,1,1,0,0,
        0,1,0,4,
        1,1,0,
        1,1,
        0
    }
    \drawPipe 
    {5/1,5/2} 
    {} 
    {1/6,2/5,3/4,4/3} 
    {1/7,2/6,3/5,4/4,5/3} 
    \drawAxes{7}
    \draw[black,dashed,very thick] (0.5,-3.5) -- (5.5,-3.5);
    \drawNorthLabels{0}{1}{1,2,3,4,5,6,{\color{\alertcolor}7}}
    \drawWestLabels {1}{0}{2,1,6,5,{\color{\alertcolor}7},4,3}

    \node at (6,-8) {$(P,p) = (P_0,p_0)$};
    \draw[->] (4,-9) -- (4,-11)
    node [midway,left] {$\phi_{B,k}$};
    \draw[->] (8,-4) -- (17,-4)
    node[midway,above] {$\Phi_B$};
    \end{shift}

    \begin{shift}[(0,-11)]
    \drawPD{7}{
        1,0,0,0,1,0,
        0,1,1,0,0,
        0,1,4,0,
        1,1,0,
        1,1,
        0
    }
    \drawPipe 
    {5/1,5/2} 
    {} 
    {1/6,2/5,3/4,4/3} 
    {1/7,2/6,3/5,4/4,5/3} 
    \drawAxes{7}
    \draw[black,dashed,very thick] (0.5,-3.5) -- (5.5,-3.5);
    
    \node at (4,-8) {$(P_1,p_1)$};
    \draw[->] (6,-4.5) -- (9,-4.5)
    node[midway, above] {$\phi_{B,k}$};
    \end{shift}

    \begin{shift}[(9,-11)]
    \drawPD{7}{
        1,0,0,0,1,0,
        0,1,4,0,0,
        1,1,0,0,
        1,1,0,
        1,1,
        0
    }
    \drawPipe 
    {5/1,5/2} 
    {} 
    {1/6,2/5,3/4,4/3} 
    {1/7,2/6,3/5,4/4,5/3} 
    \drawAxes{7}
    \draw[black,dashed,very thick] (0.5,-3.5) -- (5.5,-3.5);

    \node at (4,-8) {$(P_2,p_2)$};
    \draw[->] (6,-4.5) -- (9,-4.5)
    node[midway, above] {$\phi_{B,k}$};
    \end{shift}

    \begin{shift}[(18,-11)]
    \drawPD{7}{
        1,0,0,0,1,0,
        4,1,0,0,0,
        1,1,0,0,
        1,1,0,
        1,1,
        0
    }
    \drawPipe 
    {5/1,5/2} 
    {} 
    {1/6,2/5,3/4,4/3} 
    {1/7,2/6,3/5,4/4,5/3} 
    \drawAxes{7}
    \draw[black,dashed,very thick] (0.5,-3.5) -- (5.5,-3.5);

    \node at (4,-8) {$(P_3,p_3)$};
    \draw[->] (4,0) -- (4,2)
    node[midway, right] {$\Phi_0$};
    \end{shift}

    \begin{shift}[(18,0)]
    \drawPD{7}{
        1,0,0,0,1,0,
        1,1,0,0,0,
        1,1,0,0,
        1,1,0,
        1,1,
        0
    }
    \drawPipe 
    {5/1,5/2} 
    {} 
    {1/6,2/5,3/4,4/3} 
    {1/7,2/6,3/5,4/4,5/3} 
    \drawAxes{7}
    \drawNorthLabels{0}{1}{1,2,3,4,5,6,{\color{\alertcolor}7}}
    \drawWestLabels {1}{0}{2,4,6,5,{\color{\alertcolor}7},1,3}
    
    \node at (4,-8) {$Q$};
    \end{shift}
\end{tikzpicture}
\caption{
    Computation of $\Phi_B(P,p) = Q$ for a pair $(P,p)\in \calP_B$ ($w = 2165743$, $\pi = 6573421$, $\lambda = 544221$).
    Pipe $k = 7$ is emphasized in blue, and the cutoff row $\clambda_{\pi w^{-1}(k)} = 3$ is indicated with a dashed line.
    The final pipe dream $Q$ has permutation $t_{14}w = 2465713$.
}
\label{fig:Phi_B}
\end{figure}
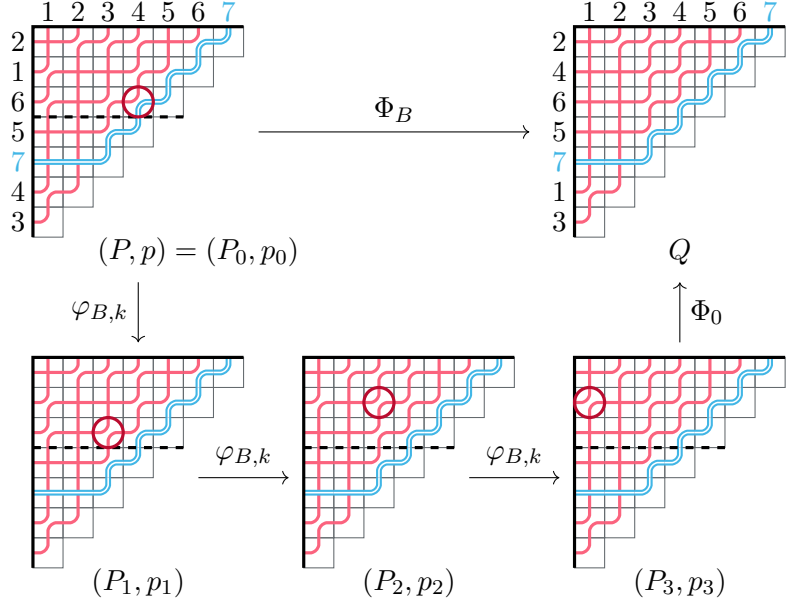

\begin{remark}
A priori, the sets $\calP_0$, $\calP_A$, and $\calP_B$ all depend on both $w$ and $\pi$, since the same is true for $\DeltaPD(w)$.
It turns out, in fact, that $\calP_A$ only depends on $w$, even though it is conditioned over $\DeltaPD(w)$.
This is because of the following proposition:
\begin{proposition}
\label{prop:P_A only cares about w}
    Let $P\in \PD(w)$ for some $w\in S_\infty$.
    If $p\in P(\bump)$ satisfies $\south_p < \west_p$, then $p\in P(\pi)$ for \textit{any} dominant $\pi \geq_L w$
\end{proposition}

\begin{proof}
Note that pipes $\south_p$ and $\west_p$ must cross above row $\row_p$.
Hence $(\south_p,\west_p)\in \Inv(w^{-1})\sseq \coInv(\pi w^{-1})$ and so $(w^{-1}(\west_p),w^{-1}(\south_p))\in \Inv(\pi)$.
It follows that $\row_p \leq w^{-1}(\west_p) \leq \clambda_{\pi w^{-1}(\south_p)}$, so we conclude $p\in P(\pi)$.
\end{proof}

\noindent
Since the definition of $(A,k)$-aligned also does not depend on $\pi$, the map $\Phi_A$ is independent of $\pi$ as well.
\end{remark}

\subsection{Auxiliary results on $(*,k)$-aligned pairs}

In this section, we will prove various results on $(*,k)$-aligned pairs which, in particular, were needed for the construction of $\Phi_A$ and $\Phi_B$ given in the previous section.

We begin with the following lemma, which was needed in defining the operators $\phi_{*,k}$ acting on $(*,k)$-aligned pairs.

\begin{lemma}
\label{lem:aligned is slidable}
Let $(P,p)$ be a $(*,k)$-aligned pair \textup{(}$*\in \{A,B\}$, $k\geq 1$\textup{)}.
Then $(P,p)$ is slidable.
\end{lemma}

\begin{proof}
Suppose that $p$ is $(A,k)$-aligned in $P$.
This is equivalent to pipes $\west_p$ and $k$ crossing at some position $q\in P(\cross)$ with $\row_q > \row_p$.
Then $\south_q = \west_p$.
Hence $w^{-1}(\west_p) > \row_q > \row_p$, so $(P,p)$ is slidable.

Now, suppose that $p$ is $(B,k)$-aligned.
Then both $(\west_p,k) \in \coInv(\pi w^{-1})$ and $(\row_p,w^{-1}(k))\in \Inv(\pi)$.
These imply that $\row_p \neq w^{-1}(\west_p)$, since otherwise we would get both $\pi w^{-1}(\west_p) < \pi w^{-1}(k)$ and $\pi w^{-1}(\west_p) > \pi w^{-1}(k)$.
Hence $\row_p > w^{-1}(\west_p)$, so $(P,p)$ is slidable.
\end{proof}

The next proposition collects various properties of the operators $\phi_{*,k}$.
Property (ii) involves the following quantities:
if $P\in \PD(w)$ and $p\in P(\bump)$, then define
\begin{align*}
    \Sigma_{A,k}(P,p) &:= \{
        x\in \bbN^2 \mid
        \text{$x$ lies (weakly) northwest of pipes $\west_{P,p}$ and $k$}
    \}, \\
    \Sigma_{B,k}(P,p) &:= \{
        x\in \bbN^2 \mid
        x\in \Sigma_{A,k}(P,p) \text{ and }
        \row_x \leq \clambda_{\pi w^{-1}(k)}
    \}.
\end{align*}

\begin{figure}
\begin{tikzpicture}[scale=0.4]
    \begin{shift}[(0,0)]
    \drawPD{7}{
        1,0,0,0,1,0,
        0,0,0,0,4,
        0,1,0,1,
        0,1,0,
        0,1,
        0
    }
    \drawPipe 
    {3/4,4/2} 
    {} 
    {1/6,2/5,3/3,4/1} 
    {1/7,2/6,3/5,4/3,5/1} 
    \drawPDShade{7}{
        1,1,1,1,3,0,
        1,1,1,1,4,
        1,1,1,4,
        1,2,4,
        4,0,
        0
    }
    \drawAxes{7}

    \node at (4,0.5) {$\Sigma_{A,k}(P_0,p_0)$};
    \node at (8,0.5) {$\subset$};
    \draw[->] (6,-4.5) -- (8,-4.5)
    node[midway,above] {$\phi_{A,k}$};
    \end{shift}

    \begin{shift}[(8,0)]
    \drawPD{7}{
        1,0,0,0,4,0,
        0,0,0,1,0,
        0,1,0,1,
        0,1,0,
        0,1,
        0
    }
    \drawPipe 
    {3/4,4/2} 
    {} 
    {1/6,2/5,3/3,4/1} 
    {1/7,2/6,3/5,4/3,5/1} 
    \drawPDShade{7}{
        1,1,1,1,4,0,
        1,1,1,3,0,
        1,1,1,4,
        1,2,4,
        4,0,
        0
    }
    \drawAxes{7}
    
    \node at (4,0.5) {$\Sigma_{A,k}(P_1,p_1)$};
    \node at (8,0.5) {$\subset$};
    \draw[->] (6,-4.5) -- (8,-4.5)
    node[midway,above] {$\phi_{A,k}$};
    \end{shift}

    \begin{shift}[(16,0)]
    \drawPD{7}{
        1,0,0,4,0,0,
        0,0,0,1,0,
        0,1,0,1,
        0,1,0,
        0,1,
        0
    }
    \drawPipe 
    {3/4,4/2} 
    {} 
    {1/6,2/5,3/3,4/1} 
    {1/7,2/6,3/5,4/3,5/1} 
    \drawPDShade{7}{
        1,1,1,4,0,0,
        1,1,4,0,0,
        1,3,0,0,
        1,4,0,
        4,0,
        0
    }
    \drawAxes{7}
    
    \node at (4,0.5) {$\Sigma_{A,k}(P_2,p_2)$};
    \node at (8,0.5) {$\subset$};
    \draw[->] (6,-4.5) -- (8,-4.5)
    node[midway,above] {$\phi_{A,k}$};
    \end{shift}

    \begin{shift}[(24,0)]
    \drawPD{7}{
        1,0,4,0,0,0,
        0,0,0,1,0,
        0,1,0,1,
        0,1,0,
        0,1,
        0
    }
    \drawPipe 
    {3/4,4/2} 
    {} 
    {1/6,2/5,3/3,4/1} 
    {1/7,2/6,3/5,4/3,5/1} 
    \drawPDShade{7}{
        1,1,4,0,0,0,
        1,4,0,0,0,
        4,0,0,0,
        0,0,0,
        0,0,
        0
    }
    \drawAxes{7}
    
    \node at (4,0.5) {$\Sigma_{A,k}(P_3,p_3)$};
    \end{shift}
\end{tikzpicture}
\caption{
    The sets $\Sigma_A(P_i,p_i)$ of the pipe dreams in Figure \ref{fig:Phi_A}.
}
\label{fig:Sigma_A}
\end{figure}
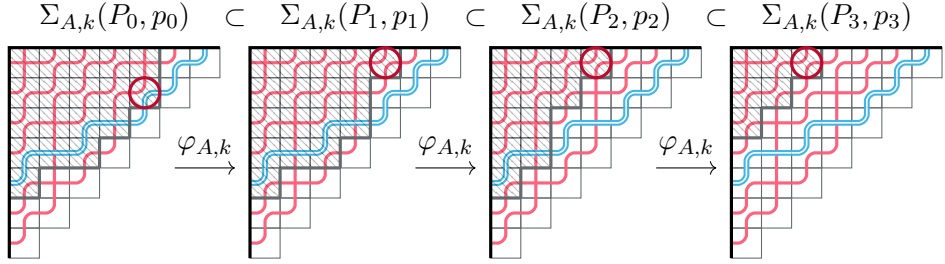

\begin{proposition}
\label{prop:phi props}
Let $(P,p)$ be a $(*,k)$-aligned pair \textup{(}$*\in \{A,B\}$, $k\geq 1$\textup{)}.
Write $(Q,q) = \phi_{*,k}(P,p)$.
\begin{enumerate}
    \item[(i)]
    $\west_{Q,q} \leq \west_{P,p}$ and $\south_{Q,q} \neq k$.
    \item[(ii)]
    $\Sigma_{*,k}(Q,q)$ is properly contained in $\Sigma_{*,k}(P,p)$.
    \item[(iii)]
    If $(Q,q)$ is not $(*,k)$-aligned, then it is in $\calP_0$.
\end{enumerate}
\end{proposition}

\begin{proof}
To simplify notation, let $(P,p') = \slide(P,p)$ and write $a = \west_{P,p'}$, $b = \south_{P,p'} = \west_{P,p}$, $c = \south_{P,p}$, and  $i = \row_p = \row_{p'}$.
Notice that $p$ lies (weakly) northwest of pipe $k$ and $b < k$ (regardless of if $*=A$ or $*=B$), so pipes $b$ and $k$ do not cross at or above row $i$.

We break the proof up into two cases depending on if $(Q,q) = (P,p')$ or $(Q,q') = \swap(P,p')$.

\textbf{Case 1.}
Suppose $(Q,q) = (P,p')$.
We first show $a < b$:
if instead $a > b$, then pipes $a$ and $b$ would cross at some $q\in P(\cross)$ with $\row_q < i$.
Since pipe $b$ is left of pipe $k$ in the first $i$ rows of $P$, $q$ must lie northwest of pipe $k$.
But then $\swap(P,p')$ would be $(*,k)$-aligned, so we should have obtained $\phi_{*,k}(P,p) = \swap(P,p')$, a contradiction.
Hence $a < b$ as desired.
Since also $\south_{P,p'} = b < k$, this gives (i).

For (ii), notice that if $* = A$, then pipe $a = \west_{P,p'}$ stays left of pipe $b = \west_{P,p}$ in the region northwest of pipe $k$, and if $* = B$, then the same is true in the part of this region at or above row $\clambda_{\pi w^{-1}(k)}$.
This shows that $\Sigma_{*,k}(P,p')\sseq \Sigma_{*,k}(P,p)$.
This containment is proper since, for instance, $p$ is in the latter but not the former.

We are left to show (iii), so assume $(P,p')$ is not $(*,k)$-aligned.
We will show $(P,p')\in \calP_0$.
First consider when $* = A$.
Since $p'$ is northwest of pipe $k$, the failure for $(P,p')$ to be $(A,k)$-aligned must be that $(a,k)\in \coInv(w^{-1})$.
Since $(P,p)$ is $(A,k)$-aligned, $(b,k)\in \Inv(w^{-1}) \sseq \coInv(\pi w^{-1})$ and so pipes $b$ and $k$ cross at some $q\in P(\cross)$ with $\row_q > i$ and $\south_{P,q} = b$.
Since $P(\cross)\sseq P(\pi)$, we get $\pi w^{-1}(b) \leq \lambda_{\row_q} \leq \lambda_i$, so $p'\in P(\pi)$.
Furthermore, we cannot have $(a,b) \in \coInv(\pi w^{-1})$---otherwise, $w^{-1}(a) < w^{-1}(k) < w^{-1}(b)$ and $\pi w^{-1}(a) < \pi w^{-1}(b) < \pi w^{-1}(k)$, so $\pi$ would contain a $132$-pattern.
Thus $(a,b)\in \Inv(\pi w^{-1})$, so we conclude $(P,p')\in \calP_0$.

We now consider what happens when $* = B$.
Here we must have $(a,k)\in \Inv(\pi w^{-1}) \sseq \coInv(w^{-1})$ for $(P,p')$ to not be $(B,k)$-aligned.
Since $(P,p)$ is $(B,k)$-aligned, $(b,k)\in \coInv(\pi w^{-1})$ and so $\pi w^{-1}(b) < \pi w^{-1}(k) < \pi w^{-1}(a)$.
This means $(a,b)\in \coInv(w^{-1})$, since if $(a,b)\in \Inv(w^{-1})$ then $w^{-1}(b) < w^{-1}(a) < w^{-1}(k)$ would give a $132$-pattern in $\pi$.
In particular, we get $(a,b)\in \Inv(\pi w^{-1})$.
Furthermore, $i\leq \clambda_\pi w^{-1}(k) \leq \clambda_{\pi w^{-1}(b)}$, so $p'\in P(\pi)$.
It follows that $(P,p')\in \calP_0$.

\textbf{Case 2.} Suppose $(Q,q) = \swap(P,p')$.
We claim that $\west_{Q,q} = b$ (and so $\south_{Q,q} = a \neq k$), giving (i).
If $a<b$, then we must have $\row_q > \row_p$, so $b = \south_{P,p'} = \west_{Q,p'} = \west_{Q,q}$.
If $a>b$, then $\row_q < \row_p$, so $b = \south_{P,q} = \west_{Q,q}$.
This gives the claim.

(iii) is vacuous here, since $(Q,q)$ is always $(*,k)$-aligned by construction.

It remains to show (ii).
When passing from $P$ to $Q$, pipe $b = \west_{P,p} = \west_{Q,q}$ changes by swapping out its segment between positions $p'$ and $q$ with the corresponding segment of pipe $a$.
Since in $P$, this segment of pipe $b$ lies below the corresponding segment of pipe $a$, and since both of these segments lie northwest of pipe $k$, we have an inclusion $\Sigma_{*,k}(Q,q)\sseq \Sigma_{*,k}(P,p)$ which is proper again by observing that $p$ is in the difference.
\end{proof}

As a consequence of Proposition \ref{prop:phi props}, we get the following corollary which was used to show that the chain of pipe dreams obtained in the construction of $\Phi_A$ and $\Phi_B$ terminates.

\begin{corollary}
\label{cor:phi terminates}
Let $(P,p)$ be a $(*,k)$-aligned pair.
Then there is some $m\geq 1$ such that $(Q,q) = \phi_{*,k}^m(P,p)$ is not $(*,k)$-aligned.
Furthermore, $(Q,q)\in \calP_0$.
\end{corollary}

\begin{proof}
If $(P_i,p_i) := \phi_{*,k}^i(P,p)$ were $(*,k)$-aligned for all $i\geq 1$, then by Proposition \ref{prop:phi props}(ii) we would have an infinite descending chain \[
    \Sigma_{*,k}(P_0,p_0) \supset \Sigma_{*,k}(P_1,p_1) \supset \cdots.
\]
Hence there must be some $m\geq 1$ such that $(Q,q) = \phi_{*,k}^m(P,p)$ is not $(*,k)$-aligned.
By Proposition \ref{prop:phi props}(iii), $(Q,q)$ must lie in $\calP_0$.
\end{proof}

\begin{figure}
\begin{tikzpicture}[scale=0.4]
    \begin{shift}[(0,0)]
    \drawPD{7}{
        1,0,0,0,1,0,
        0,1,1,0,0,
        0,1,0,4,
        1,1,0,
        1,1,
        0
    }
    \drawPipe 
    {5/1,5/2} 
    {} 
    {1/6,2/5,3/4,4/3} 
    {1/7,2/6,3/5,4/4,5/3} 
    \drawPDShade{7}{
        1,1,1,1,3,0,
        1,1,1,1,4,
        2,2,2,4,
        0,0,0,
        0,0,
        0
    }
    \drawAxes{7}
    \draw[black,dashed,very thick] (0.5,-3.5) -- (5.5,-3.5);

    \node at (4,0.5) {$\Sigma_{B,k}(P_0,p_0)$};
    \node at (8,0.5) {$\subset$};
    \draw[->] (6,-4.5) -- (8,-4.5)
    node[midway,above] {$\phi_{B,k}$};
    \end{shift}

    \begin{shift}[(8,0)]
    \drawPD{7}{
        1,0,0,0,1,0,
        0,1,1,0,0,
        0,1,4,0,
        1,1,0,
        1,1,
        0
    }
    \drawPipe 
    {5/1,5/2} 
    {} 
    {1/6,2/5,3/4,4/3} 
    {1/7,2/6,3/5,4/4,5/3} 
    \drawPDShade{7}{
        1,1,1,4,0,0,
        1,1,3,0,0,
        2,2,4,0,
        0,0,0,
        0,0,
        0
    }
    \drawAxes{7}
    \draw[black,dashed,very thick] (0.5,-3.5) -- (5.5,-3.5);

    \node at (4,0.5) {$\Sigma_{B,k}(P_1,p_1)$};
    \node at (8,0.5) {$\subset$};
    \draw[->] (6,-4.5) -- (8,-4.5)
    node[midway,above] {$\phi_{B,k}$};
    \end{shift}

    \begin{shift}[(16,0)]
    \drawPD{7}{
        1,0,0,0,1,0,
        0,1,4,0,0,
        1,1,0,0,
        1,1,0,
        1,1,
        0
    }
    \drawPipe 
    {5/1,5/2} 
    {} 
    {1/6,2/5,3/4,4/3} 
    {1/7,2/6,3/5,4/4,5/3} 
    \drawPDShade{7}{
        1,1,1,4,0,0,
        1,2,4,0,0,
        4,0,0,0,
        0,0,0,
        0,0,
        0
    }
    \drawAxes{7}
    \draw[black,dashed,very thick] (0.5,-3.5) -- (5.5,-3.5);

    \node at (4,0.5) {$\Sigma_{B,k}(P_2,p_2)$};
    \node at (8,0.5) {$\subset$};
    \draw[->] (6,-4.5) -- (8,-4.5)
    node[midway,above] {$\phi_{B,k}$};
    \end{shift}

    \begin{shift}[(24,0)]
    \drawPD{7}{
        1,0,0,0,1,0,
        4,1,0,0,0,
        1,1,0,0,
        1,1,0,
        1,1,
        0
    }
    \drawPipe 
    {5/1,5/2} 
    {} 
    {1/6,2/5,3/4,4/3} 
    {1/7,2/6,3/5,4/4,5/3} 
    \drawPDShade{7}{
        3,0,0,0,0,0,
        4,0,0,0,0,
        0,0,0,0,
        0,0,0,
        0,0,
        0
    }
    \drawAxes{7}
    \draw[black,dashed,very thick] (0.5,-3.5) -- (5.5,-3.5);

    \node at (4,0.5) {$\Sigma_{B,k}(P_3,p_3)$};
    \end{shift}
    
\end{tikzpicture}
\caption{
    The sets $\Sigma_B(P_i,p_i)$ of the pipe dreams in Figure $\ref{fig:Phi_B}$.
}
\label{fig:Sigma_B}
\end{figure}
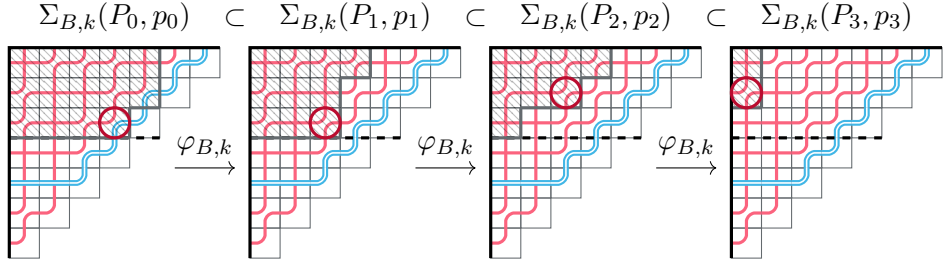

\subsection{Properties of $\Phi$}

To complete our proof of Theorem \ref{thm:delta pi}, it suffices to show that $\Phi$ satisfies ($\Phi$.1z), ($\Phi$.1a), ($\Phi$.1b), and ($\Phi$.2).
We first take care of the last of these.

\begin{proposition}
\label{prop:Phi weights}
Let $(P,p)\in \DeltaPD(w)$ and write $Q = \Phi(P,p)$.
Then $\frac{x_{\row_p}}{y_{\row_p}} x^{P(\smallcross)} y^{P(\smallbump)\cap P(\pi)}
= x^{Q(\smallcross)} y^{Q(\smallbump)\cap P(\pi)}$.
\end{proposition}

\begin{proof}
If $(P,p)\in \calP_0$, then the statement is clear.
If $(P,p)\in \calP_A\cup \calP_B$, then $\Phi(P,p)$ is obtained from $(P,p)$ by applying some sequence of $\slide$ and $\swap$ moves and then applying $\Phi_0$.
Hence it suffices to show that both $\slide$ and $\swap$ preserve the monomial weighting on pairs $(P,p)$ given by $\frac{x_{\row_p}}{y_{\row_p}} x^{P(\smallcross)} y^{P(\smallbump)\cap P(\pi)}$,
but this is clear from their definitions.
\end{proof}

The remainder of this paper will be dedicated to proving the three parts constituting ($\Phi$.1).
We start by showing ($\Phi$.1z).

\begin{proposition}
\label{prop:Phi_0}
The map \[
    \Phi_0\colon \calP_0\to \bigcup_{w\lessdot_\pi t_{ab}w} \PD(t_{ab}w)
\] is a bijection.
\end{proposition}

\begin{proof}
Suppose we are given $R\in \PD(t_{ab}w)$ for some $w,t_{ab}w\leq_L \pi$.
Pipes $a$ and $b$ must cross in $R$ at some position $p$, so we can replace this $\cross$ with a $\bump$.
Since $p\in R(\cross)$, we have $\row_p \leq \clambda_{\pi (t_{ab}w)^{-1}(\south_{R,p})} = \clambda_{\pi w^{-1}t_{ab}(a)} = \clambda_{\pi w^{-1}(b)} = \clambda_{\pi w^{-1}(\south_{P,p})}$, so indeed $p\in P(\pi)$ and $(P,p)\in \DeltaPD$.
Lemma \ref{prop:cover lemma} tells us that $(a,b)\in \Inv(\pi w^{-1})$, therefore $(P,p)\in \calP_0$.
This procedure clearly gives an inverse to $\Phi_0$.
\end{proof}

It remains to show ($\Phi$.1a) and ($\Phi$.1b).
These will be granted by the next proposition:

\begin{proposition}
\label{prop:Phi_A/B bijection}
Let $Q\in \PD(t_{ab}w)$ for some $w\lessdot_\pi t_{ab}w$.
Then the maps \[
    \Phi_A^{-1}(Q)\to A(w,t_{ab}w)
    \qquad\text{and}\qquad
    \Phi_B^{-1}(Q)\to B(w,t_{ab}w)
\] given by sending a pair $(P,p)$ to the value $k$ determined in Definition \ref{def:Phi_A/B} are bijections.
\end{proposition}

In turn, Proposition \ref{prop:Phi_A/B bijection} will follow from two lemmas, the first showing that the above map is well-defined:

\begin{lemma}
\label{lem:Phi_A/B image}
With the hypotheses and notation from Definition \ref{def:Phi_A/B}, we have $k\in *(w,t_{ab}w)$.
\end{lemma}

\begin{proof}
For $(P_m,p_m)$ to not be $(*,k)$-aligned, we must have that $(P_m,p_m) = \slide(P_{m-1},p_{m-1})$, so $P_m = P_{m-1}$ and $p_m$ is the first bump due west of $p_{m-1}$.
In particular, we get $a < b < k$, since $b = \west_{P_m,p_{m-1}}$ and $(P_m,p_{m-1})$ is $(*,k)$-aligned.

Suppose $(P,p)\in \calP_A$.
Notice that $(P_m,p_{m-1})$ being $(A,k)$-aligned tells us that $(b,k)\in \Inv(w^{-1})$.
Also, $(P_m,p_m)$ not being $(A,k)$-aligned means that $(a,k)\in \coInv(w^{-1})$ (since $p_m$ is northwest of pipe $k$).
Putting these together, we get $a < b < k$ and $w^{-1}(a) < w^{-1}(k) < w^{-1}(b)$, so $k\in A(w,t_{ab}w)$.

Suppose $(P,p)\in \calP_B$.
Since $(P_m,p_{m-1})$ is $(B,k)$-aligned, $(b,k)\in \coInv(\pi w^{-1})$.
Since $(P_m,p_m)$ is not $(B,k)$-aligned, $(a,k)\in \Inv(\pi w^{-1})$ (because $p_m$ is northwest of pipe $k$ and $\row_{p_m} = \row_{p_{m-1}} \leq \clambda_j$).
Thus $a < b < k$ and $\pi w^{-1}(b) < \pi w^{-1}(k) < \pi w^{-1}(a)$, so $k\in B(w,t_{ab}w)$.
\end{proof}

The next lemma forms the heart of Proposition \ref{prop:Phi_A/B bijection} by allowing us to construct an inverse to each map $\Phi_*^{-1}(Q) \to *(w,t_{ab}w)$.

\begin{lemma}
\label{lem:Phi_A/B inverse}
Let $*\in \{A,B\}$ and $k\geq1$.
Suppose that $Q\in \PD(w)$ and $q\in Q(\bump)$ satisfy one of the following conditions holds \textup{(}writing $a = \west_{Q,q}$, $b = \south_{Q,q}$\textup{)}:
\begin{itemize}
    \item
    $(Q,q)\in \calP_0$ and $k\in *(w,t_{ab}w)$.
    \item
    $(Q,q)$ is $(*,k)$-aligned and $b\neq k$.
\end{itemize}
Then there is a unique $(*,k)$-aligned pair $(P,p)$ such that $\phi_{*,k}(P,p) = (Q,q)$.
\end{lemma}

\begin{proof}
Write $i = \row_q$ and $j = \pi w^{-1}(k)$.
We break up the proof of the lemma into two parts depending on which condition $(Q,q)$ satisfies.

\textbf{Condition 1.}
Suppose $(Q,q)\in \calP_0$ and $k\in *(w,t_{ab}w)$.
We claim that both $(a,b),(a,k)\in \coInv(w^{-1})$, meaning pipe $a$ crosses neither pipe $b$ nor pipe $k$ in $Q$.
This is clear if $* = A$, since $a < b < k$ 
and $w^{-1}(a) < w^{-1}(k) < w^{-1}(b)$.
If $* = B$, so $a < b < k$ and $\pi w^{-1}(b) < \pi w^{-1}(k) < \pi w^{-1}(a)$, then this follows from $(a,b),(a,k)\in \Inv(\pi w^{-1})\sseq \coInv(w^{-1})$.

Pipe $k$ does not cross pipe $b$ in the first $i$ rows---if it did, then it would necessarily cross pipe $a$ as well (since pipes $a$ and $b$ bump at $q$).
Hence if $p\in Q(\bump)$ denotes the bump east of $q$ with $\west_{Q,p} = b$ and $\row_p = i$, then $p$ lies northwest of pipe $k$.
Since $(b,k)$ lies in $\Inv(w^{-1})$ (when $* = A$) or $\coInv(\pi w^{-1})$ (when $* = B$), $(Q,p)$ is $(*,k)$-aligned.
Furthermore, $\phi_{*,k}(Q,p) = \slide(Q,p) = (Q,q)$ since pipes $a$ and $b$ do not cross.

For uniqueness, note that if $(Q,q) = \phi_{*,k}(\tilde P,\tilde p)$ for a $(*,k)$-aligned pair $(\tilde P, \tilde p)$, then we must have $(Q,q) = \slide(\tilde P, \tilde p)$ since pipes $a$ and $b$ do not cross.
Since $\slide$ is 1-to-1, we conclude $(\tilde P,\tilde p) = (P,p)$.

\textbf{Condition 2.}
Suppose $(Q,q)$ is $(*,k)$-aligned and $b\neq k$.
Then either $(a,k)\in \Inv(w^{-1})\sseq \coInv(\pi w^{-1})$ (if $* = A$) or $(a,k)\in \coInv(\pi w^{-1})$ (if $* = B$).
Either way $a < k$, so pipes $a$ and $k$ do not cross at or above row $i$.

We further break-up the proof of this part into three separate cases.

\textbf{Case 1.}
Suppose that $(Q,q)$ is swappable and, writing $(P,p) := \swap(Q,q)$, that $p$ is northwest of pipe $k$.
Further assume that $\row_p \leq \clambda_j$ if $* = B$.
Let $p'$ denote the position obtained as follows:
\begin{itemize}
    \item
    If $\row_p < i$, then $\west_{Q,p} = a$.
    Since pipes $a$ and $k$ do not cross in row $\row_p$, there must be some $p'\in Q(\bump)$ northwest of pipe $k$ with $\west_{Q,p'} = a$ and $\row_{p'} = \row_p$.
    \item
    If $\row_p > i$, then $\west_{Q,p} = b$.
    Since $q$ and $p$ are each northwest of pipe $k$, pipe $k$ cannot cross the segment of pipe $b$ from $q$ to $p$.
    Hence there is some $p'\in Q(\bump)$ northwest of pipe $k$ with $\west_{Q,p'} = b$ and $\row_{p'} = \row_p$.
\end{itemize}
In either case, it can be checked that $\west_{P,p'} = \south_{P,p} = a$, so $(P,p')$ is $(*,k)$-aligned.
Furthermore, $\slide(P,p') = (P,p)$ is swappable and $\swap(P,p) = (Q,q)$ is $(*,k)$-aligned, so $\phi_{*,k}(P,p') = (Q,q)$.

For uniqueness, suppose that $(Q,q) = \slide(Q,\tilde p)$ for some $\tilde p\in Q(\bump)$ which is $(*,k)$-aligned in $Q$.
Then $\west_{Q,\tilde p} = b = \west_{P,p}$.
Since $(Q,\tilde p)$ is $(*,k)$-aligned, $(\west_{P,p},k) = (\west_{Q,\tilde p},k)$ lies in $\Inv(w^{-1})$ (for $* = A$) or $\coInv(\pi w^{-1})$) (for $* = B$).
This means $(P,p)$ is $(*,k)$-aligned, so $\phi_{*,k}(Q,\tilde p) = (P,p) \neq (Q,q)$: a contradiction.
Hence if $(Q,q) = \phi_{*,k}(\tilde P,\tilde p)$ for a $(*,k)$-aligned pair $(\tilde P,\tilde p)$, then we must have $(Q,q) = \swap(\slide(\tilde P,\tilde p))$.
Since $\slide$ and $\swap$ are each 1-to-1, we conclude $(\tilde P,\tilde p) = (P,p')$.

\textbf{Case 2.}
Suppose that $(Q,q)$ is swappable but, writing $(P,p) := \swap(Q,q)$, either $p$ is not northwest of pipe $k$, or $* = B$ and $\row_p > \clambda_j$.
This implies $a < b$, since otherwise $\row_p < i$ and we know pipe $a$ stays left of pipe $k$ in the first $i$ rows.

We claim that we also have $b < k$.
Suppose instead that $b > k$.
Then pipes $b$ and $k$ must cross above row $i$, since $q$ lies northwest of pipe $k$.
Then $(k,b)\in \Inv(w^{-1})$ and so $(w^{-1}(b),w^{-1}(k))\in \Inv(w)\sseq \Inv(\pi)$.
On the other hand, pipe $b$ stays left of pipe $k$ in rows below $i$ and $\row_p > i$, so $p$ is northwest of pipe $k$.
By hypothesis, we must have $* = B$ and $\row_p > \clambda_j$.
In particular, $w^{-1}(b) \geq \row_p > \clambda_j$, so $(w^{-1}(b),j)\notin D(\pi)$ (since $\pi$ is dominant).
This implies $(w^{-1}(b),\pi^{-1}(j)) = (w^{-1}(b),w^{-1}(k))\notin \Inv(\pi)$ which contradicts our previous observation.

To recap, we have shown that $a < b < k$.
In this scenario, pipes $b$ and $k$ do not cross at or above row $i$, so pipe there is some $q'\in Q(\bump)$ northwest of pipe $k$ with $\west_{Q,q'} = b$ and $\row_{q'} = i$.

We are left to show that $q'$ is $(*,k)$-aligned in $Q$---given this, since $\swap(Q,q) = (P,p)$ is not $(*,k)$-aligned, we would get $\phi_{*,k}(Q,q') = \slide(Q,q') = (Q,q)$.
So, for the sake of contradiction, suppose that $q'$ is not $(*,k)$-aligned in $Q$.
\begin{itemize}
\item 
If $* = A$, then this means $(b,k)\in \coInv(w^{-1})$, so pipes $b$ and $k$ do not cross in $Q$.
Since $(Q,q)$ is $(A,k)$-aligned, $(a,k)\in \Inv(w^{-1})$ and so pipes $a$ and $k$ do cross in $Q$.
But $a < b < k$, so this is absurd.
\item 
If $* = B$, then we must have $(b,k)\in \Inv(\pi w^{-1}) \sseq \coInv(w^{-1})$.
Then $(w^{-1}(b),w^{-1}(k))\in \Inv(\pi)$ so $(w^{-1}(b),\pi w^{-1}(k)) = (w^{-1}(b),j)\in D(\pi)$.
Since $\pi$ is dominant, $w^{-1}(b)\leq \clambda_j$.
On the other hand, $p$ lies on pipe $b$, so $w^{-1}(b)\geq \row_p > \clambda_j$, giving a contradiction.
\end{itemize}

For uniqueness, suppose $(Q,q) = \swap(\slide(P,\tilde p))$ for some $\tilde p\in P(\bump)$ which is $(*,k)$-aligned in $P$.
This means $\slide(P,\tilde p) = (P,p)$.
If $p$ is not northwest of pipe $k$, or (when $* = B$) if $\row_p > \clambda_j$, then the same would be true of $\tilde p$.
This contradicts that $\tilde p$ is $(*,k)$-aligned in $P$, so no such $\tilde p$ exists.
Hence if $(Q,q) = \phi_{*,k}(\tilde P,\tilde p)$ for some $(*,k)$-aligned pair $(\tilde P,\tilde p)$, then $(Q,q) = \slide(\tilde P,\tilde p)$.
Again, this forces $(\tilde P,\tilde p) = (Q,q')$.

\textbf{Case 3.}
Finally, suppose that $(Q,q)$ is not swappable, so pipes $a$ and $b$ do not cross in $Q$.
We claim that $a < b < k$, as in the previous case.
Given this, we would then be able to find some $q'\in Q(\bump)$ northwest of pipe $k$ with $\west_{Q,q'} = b$ and $\row_{q'} = i$.
The verbatim argument from the previous case would then show that $q'$ is $(*,k)$-aligned in $Q$.
Since $\slide(Q,q') = (Q,q)$ is not swappable, we would then conclude $\phi_{*,k}(Q,q') = (Q,q)$.

We now show the claim.
The argument that $a < b$ is the same as one given before.
If $b > k$, then (as before) pipes $b$ and $k$ must cross above row $i$.
In particular, $(k,b)\in \Inv(w^{-1})\sseq \coInv(\pi w^{-1})$.
We know pipe $k$ does not cross pipe $a$ at or above row $i$, and it also cannot cross pipe $a$ below row $i$ since it would need to cross pipe $b$ again to do so.
Hence pipes $a$ and $k$ do not cross i.e. $(a,k)\in \coInv(w^{-1})$.
If $* = A$, this contradicts that $(a,k)\in \Inv(w^{-1})$.
If $* = B$, the contradiction arises from observing that $(a,k)\in \coInv(\pi w^{-1})$ implies
both $w^{-1}(a) < w^{-1}(b) < w^{-1}(k)$ and $\pi w^{-1}(a) < \pi w^{-1}(k) < \pi w^{-1}(b)$, which gives an occurrence of a $132$-pattern in $\pi$.
The claim follows.

Since pipes $a$ and $b$ do not cross, uniqueness holds as it did for Condition 1.
This completes the proof.
\end{proof}

We may now give the proof of Proposition \ref{prop:Phi_A/B bijection}, completing the proof of Theorem \ref{thm:delta pi}.

\begin{proof}[Proof of Proposition \ref{prop:Phi_A/B bijection}]
Suppose we are given $k\in *(w,t_{ab}w)$.
Let $(Q_0,q_0) = \Phi_0^{-1}(Q)$.
Using Lemma \ref{lem:Phi_A/B inverse}, we can work backwards to form a chain \[
    (Q_m,q_m)
    \overset{\phi_{*,k}}\longrightarrow
    \cdots
    \overset{\phi_{*,k}}\longrightarrow
    (Q_0,q_0)
    \overset{\phi_{*,k}}\longrightarrow
    (Q_0,q_0) = 
    \Phi_0^{-1}(Q)
\] which must terminate (by again considering the sets $\Sigma_k(Q_i,q_i)$) at some $(Q_m,q_m)\notin \calP_0$ which is $(*,k)$-aligned and satisfies $\south_{Q_m,q_m} = k$.
Write $c = \west_{Q_m,q_m}$.
Let $(P,p)$ denote the pair obtained as follows:
\begin{itemize}
\item
If $* = A$, set $(P,p) := \swap(Q_m,q_m)$, which is well-defined since $(c,k)\in \Inv(w^{-1})$. 
Notice that $\row_p > \row_{q_m}$ and so $\south_{P,p} = c < \west_{P,p} = k$.
By Proposition \ref{prop:P_A only cares about w}, we get $(P,p)\in \calP_A$.

\item
If $* = B$, set $(P,p) := (Q_m,q_m)$.
Notice that $(\west_{P,p},\south_{P,p}) = (c,k) \in \coInv(\pi w^{-1})$ and $\row_p \leq \clambda_{\pi w^{-1}(k)} = \clambda_{\pi w^{-1}(\south_{P,p})}$ since $(P,p)$ is $(B,k)$-aligned, so $(P,p)\in \calP_B$.
\end{itemize}

It is clear that $\Phi_A(P,p) = Q$, so we obtain a map $*(w,t_{ab}w)\to \Phi_*^{-1}(Q)$ by sending $k$ to the pair $(P,p)$.
The uniqueness statement in Lemma \ref{lem:Phi_A/B inverse} together with Proposition \ref{prop:phi props}(i) shows that this procedure gives an inverse to the map $\Phi_*^{-1}(Q)\to *(w,t_{ab}w)$.
\end{proof}

\section*{Acknowledgement}
The author would like to thank Dave Anderson for his guidance on the writing and organization of this paper.
Additional thanks is given to Yibo Gao, Zachary Hamaker, and Oliver Pechenik for helpful conversations about the differential operators studied in this paper.

\nocite{*}
\bibliographystyle{plain}
\bibliography{main}

\end{document}